\documentclass[12pt]{article}
\usepackage{amsmath,amssymb,amsthm}

\newtheorem{theorem}{Theorem}
\newtheorem{lemma}[theorem]{Lemma}
\newtheorem{corollary}[theorem]{Corollary}

\def\barr{\begin{array}}
\def\earr{\end{array}}

\title{The Chermak-Delgado lattice of ZM-groups}
\author{Marius T\u arn\u auceanu}
\date{November 13, 2016}

\begin{document}

\maketitle

\begin{abstract}
    In this note we prove that the Chermak-Delgado lattice of a ZM-group is a chain of length $0$.
    A similar conclusion is obtained for all dihedral groups $D_{2m}$ with $m\neq 4$.
\end{abstract}

{\small
\noindent
{\bf MSC2000\,:} Primary 20D30; Secondary 20D60, 20D99.

\noindent
{\bf Key words\,:} Chermak-Delgado measure, Chermak-Delgado lattice, Chermak-Delgado subgroup, subgroup lattice, ZM-group, dihedral group.}

\section{Introduction}

Let $G$ be a finite group and $L(G)$ be the subgroup lattice of $G$. The \textit{Chermak-Delgado measure} of a subgroup $H$ of $G$ is defined by
\begin{equation}
m_G(H)=|H||C_G(H)|.\nonumber
\end{equation}Let
\begin{equation}
m(G)={\rm max}\{m_G(H)\mid H\leq G\} \mbox{ and } {\cal CD}(G)=\{H\leq G\mid m_G(H)=m(G)\}.\nonumber
\end{equation}Then the set ${\cal CD}(G)$ forms a modular self-dual sublattice of $L(G)$,
which is called the \textit{Chermak-Delgado lattice} of $G$. It was first introduced by Chermak and Delgado \cite{6}, and revisited by Isaacs \cite{8}. In the last years
there has been a growing interest in understanding this lattice, especially for $p$-groups (see e.g. \cite{1,3,4,10}). The study can be naturally extended to nilpotent groups,
since by \cite{3} the Chermak-Delgado lattice of a direct product of finite groups decomposes as the direct product of the Chermak-Delgado lattices of the factors. Recall also
two other important properties of the Chermak-Delgado lattice that will be used in our paper:
\begin{itemize}
\item[-] if $H\in {\cal CD}(G)$, then $C_G(H)\in {\cal CD}(G)$ and $C_G(C_G(H))=H$;
\item[-] the minimum subgroup $M(G)$ of ${\cal CD}(G)$ (called the \textit{Chermak-Delgado subgroup} of $G$) is characteristic, abelian, and contains $Z(G)$.
\end{itemize}

In what follows we will focus on describing the Chermak-Delgado lattice of a ZM-group, that is a finite group with all
Sylow subgroups cyclic. By \cite{7} such a group is of type
\begin{equation}
{\rm ZM}(m,n,r)=\langle a, b \mid a^m = b^n = 1,
\hspace{1mm}b^{-1} a b = a^r\rangle, \nonumber
\end{equation}
where the triple $(m,n,r)$ satisfies the conditions
\begin{equation}
{\rm gcd}(m,n)={\rm gcd}(m,r-1)=1 \mbox{ and } r^n
\equiv 1 \hspace{1mm}({\rm mod}\hspace{1mm}m). \nonumber
\end{equation}It is clear that $|{\rm ZM}(m,n,r)|=mn$ and $Z({\rm ZM}(m,n,r))=\langle b^d\rangle$, where $d$ is the multiplicative order of
$r$ modulo $m$, i.e.
\begin{equation}
d={\rm min}\{k\in\mathbb{N}^* \mid r^k\equiv 1 \hspace{1mm}({\rm mod} \hspace{1mm}m)\}.\nonumber
\end{equation}The subgroups of ${\rm ZM}(m,n,r)$ have been completely described
in \cite{5}. Set
\begin{equation}
L=\left\{(m_1,n_1,s)\in\mathbb{N}^3 \hspace{1mm}\mid\hspace{1mm}
m_1|m,\hspace{1mm} n_1|n,\hspace{1mm} s<m_1,\hspace{1mm}
m_1|s\frac{r^n-1}{r^{n_1}-1}\right\}.\nonumber
\end{equation}Then there is a bijection between $L$ and the subgroup lattice
$L({\rm ZM}(m,n,r))$ of ${\rm ZM}(m,n,r)$, namely the function
that maps a triple $(m_1,n_1,s)\in L$ into the subgroup $H_{(m_1,n_1,s)}$ defined by
\begin{equation}
H_{(m_1,n_1,s)}=\bigcup_{k=1}^{\frac{n}{n_1}}\alpha(n_1,
s)^k\langle a^{m_1}\rangle=\langle a^{m_1},\alpha(n_1, s)\rangle,\nonumber
\end{equation}where $\alpha(x, y)=b^xa^y$, for all $0\leq x<n$ and $0\leq y<m$.
Note that:
\begin{itemize}
\item[-] $|H_{(m_1,n_1,s)}|=\frac{mn}{m_1n_1}$\,, for any $s$ satisfying $(m_1,n_1,s)\in L$;
\item[-] $H_{(m_1,n_1,s)}$ is normal in ${\rm ZM}(m,n,r)$ if and only if $m_1\mid r^{n_1}-1$ and $s=0$ (see \cite{9});
\item[-] two subgroups of ${\rm ZM}(m,n,r)$ are conjugate if and only if they have the same order (see \cite{2}).
\end{itemize}

We are now able to give our main result.\newpage

\begin{theorem}\label{th:C1}
    Under the above notation, we have
    \begin{equation}
    m({\rm ZM}(m,n,r))=\frac{m^2n^2}{d^2}\,.\nonumber
    \end{equation}Moreover, ${\cal CD}({\rm ZM}(m,n,r))$ is a chain of length $0$, namely
    \begin{equation}
    {\cal CD}({\rm ZM}(m,n,r))=\{H_{(1,d,0)}\}.\nonumber
\end{equation}
\end{theorem}

By taking an odd integer $m\geq 3$, $n=2$ and $r=m-1$ one obtains $d=2$ and ${\rm ZM}(m,n,r)\cong D_{2m}$. Thus, Theorem 1 leads to the following corollary.

\begin{corollary}
    If $m\geq 3$ is odd, then
    \begin{equation}
    m(D_{2m})=m^2.\nonumber
    \end{equation}Moreover, ${\cal CD}(D_{2m})$ is a chain of length $0$, namely
    \begin{equation}
    {\cal CD}(D_{2m})=\{\langle a\rangle\}.\nonumber
    \end{equation}
\end{corollary}

Now, let us assume that $m=2^km'$ with $k\geq 1$ and $m'$ odd. For $m'\geq 3$ the conclusion of Corollary 2 remains valid because $D_{2m}\cong\mathbb{Z}_2^k\times D_{2m'}$, and therefore
\begin{equation}
{\cal CD}(D_{2m})\cong{\cal CD}(\mathbb{Z}_2^k)\times {\cal CD}(D_{2m'})=\{\mathbb{Z}_2^k\}\times {\cal CD}(D_{2m'})\nonumber
\end{equation}is a chain of length $0$. We infer that the computation of ${\cal CD}(D_{2m})$ is completed by studying the case $m'=1$, i.e. for
\begin{equation}
D_{2^{k+1}}=\langle a, b \mid a^{2^k} = b^2 = 1,
\hspace{1mm}b^{-1} a b = a^{-1}\rangle. \nonumber
\end{equation}We easily obtain
\begin{equation}
    m(D_{2^{k+1}})=2^{2k}, \, \forall\, k\geq 2\nonumber
\end{equation}and
\begin{equation}
    {\cal CD}(D_{2^{k+1}})=\left\{\barr{lll}
    \{D_8, \langle a\rangle, \langle a^2, b\rangle, \langle a^2, ab\rangle, \langle a^2\rangle\},&k=2\\
    \{\langle a\rangle\},&k\geq 3.\earr\right.\nonumber
\end{equation}
\smallskip

Finally, we remark that ${\cal CD}(G)$ is a chain of length $0$ for large classes of groups $G$, such us abelian groups,
ZM-groups or dihedral groups $D_{2m}$ with $m\neq 4$. This leads to the following natural question.

\bigskip\noindent{\bf Open problem.} Which are the finite groups $G$ such that ${\cal CD}(G)=\{M(G)\}$?

\section{Proof of the main result}

We start by proving two auxiliary results.

\begin{lemma}
    For every $(m_1,n_1,s)\in L$, we have
    \begin{equation}
    m_{{\rm ZM}(m,n,r)}(H_{(m_1,n_1,s)})=\frac{mn^2{\rm gcd}(m,r^{n_1}-1)}{m_1n_1{\rm ord}_{\frac{m}{m_1}}(r)}\,,\nonumber
    \end{equation}where ${\rm ord}_{\frac{m}{m_1}}(r)$ denotes the multiplicative order of $r$ modulo $\frac{m}{m_1}$\,.
\end{lemma}

\begin{proof}
   First of all, we observe that under the notation in Section 1 we have
   \begin{equation}
   \alpha(x_1,y_1)\alpha(x_2, y_2)=\alpha(x_1+x_2, r^{x_2}y_1+y_2).\nonumber
   \end{equation}This implies that
   \begin{equation}
   \alpha(x,y)^k=b^{kx}a^{y\frac{r^{kx}-1}{r^x-1}}, \mbox{ for all } k\in\mathbb{Z},\nonumber
   \end{equation}and so
   \begin{equation}
   \alpha(x,y)^{-1}=\alpha(-x,-r^{-x}y).\nonumber
   \end{equation}Then
   \begin{equation}
   \alpha(x,y)^{-1}\alpha(u,v)\alpha(x,y)=\alpha(u,-r^uy+r^xv+y).\nonumber
   \end{equation}

   Now, let $(m_1,n_1,s)\in L$. In order to compute $|C_{{\rm ZM}(m,n,r)}(H_{(m_1,n_1,s)})|$ we can assume $s=0$, because $H_{(m_1,n_1,s)}$ and $H_{(m_1,n_1,0)}$ are conjugate. We obtain $\alpha(x,y)\in C_{{\rm ZM}(m,n,r)}(H_{(m_1,n_1,0)})$ if and only if
   \begin{equation}
   \alpha(x,y)^{-1}\alpha(0,m_1)\alpha(x,y)=\alpha(0,m_1) \mbox{ and } \alpha(x,y)^{-1}\alpha(n_1,0)\alpha(x,y)=\alpha(n_1,0),\nonumber
   \end{equation}which means
   \begin{equation}
   \alpha(0,r^xm_1)=\alpha(0,m_1) \mbox{ and } \alpha(n_1,-r^{n_1}y+y)=\alpha(n_1,0),\nonumber
   \end{equation}i.e.
   \begin{equation}
   \frac{m}{m_1}\mid r^x-1 \mbox{ and } m\mid y(r^{n_1}-1).\nonumber
   \end{equation}Clearly, these relations are equivalent respectively with
   \begin{equation}
   {\rm ord}_{\frac{m}{m_1}}(r)\mid x \mbox{ and } \frac{m}{{\rm gcd}(m,r^{n_1}-1)}\mid y,\nonumber
   \end{equation}that is $C_{{\rm ZM}(m,n,r)}(H_{(m_1,n_1,0)})=H_{(m'_1,n'_1,0)}$, where
   \begin{equation}
   m'_1=\frac{m}{{\rm gcd}(m,r^{n_1}-1)} \mbox{ and } n'_1={\rm ord}_{\frac{m}{m_1}}(r).\nonumber
   \end{equation}In particular, we have
   \begin{equation}
   |C_{{\rm ZM}(m,n,r)}(H_{(m_1,n_1,s)})|=|H_{(m'_1,n'_1,0)}|=\frac{mn}{m'_1n'_1}=\frac{n{\rm gcd}(m,r^{n_1}-1)}{{\rm ord}_{\frac{m}{m_1}}(r)}\,,\nonumber
   \end{equation}and consequently
   \begin{equation}
    m_{{\rm ZM}(m,n,r)}(H_{(m_1,n_1,s)})=m_{{\rm ZM}(m,n,r)}(H_{(m_1,n_1,0)})=\frac{mn^2{\rm gcd}(m,r^{n_1}-1)}{m_1n_1{\rm ord}_{\frac{m}{m_1}}(r)}\,,\nonumber
    \end{equation}as desired.
\end{proof}

The computation of the maximum value of $m_{{\rm ZM}(m,n,r)}(H_{(m_1,n_1,s)})$ when $m_1\mid m$ and $n_1\mid n$ is difficult by using only Lemma 3.
In order to do this the following step is crucial.

\begin{lemma}
    The Chermak-Delgado subgroup of ${\rm ZM}(m,n,r)$ is $H_{(1,d,0)}$.
\end{lemma}

\begin{proof}
    Let $H_{(m_0,n_0,s)}=M({\rm ZM}(m,n,r))$. Then the triple $(m_0,n_0,s)$ satisfies the following conditions:
    \begin{itemize}
    \item[{\rm a)}] $m_0\mid r^{n_0}-1$ and $s=0$, because $H_{(m_0,n_0,s)}$ is normal in ${\rm ZM}(m,n,r)$;
    \item[{\rm b)}] $\frac{m}{m_0}\mid r^{n_0}-1$, because $H_{(m_0,n_0,s)}$ is abelian;
    \item[{\rm c)}] $n_0\mid d$, because $Z({\rm ZM}(m,n,r))\subseteq H_{(m_0,n_0,s)}$.
    \end{itemize}Also, if $H_{(m'_0,n'_0,0)}=C_{{\rm ZM}(m,n,r)}(H_{(m_0,n_0,s)})$, then we have
    \begin{equation}
    m'_0=\frac{m}{{\rm gcd}(m,r^{n_0}-1)} \mbox{ and } n'_0={\rm ord}_{\frac{m}{m_0}}(r)\nonumber
    \end{equation}by Lemma 3. Since $C_{{\rm ZM}(m,n,r)}(H_{(m'_0,n'_0,0)})=H_{(m_0,n_0,s)}$ we infer that
    \begin{equation}
    \frac{m}{m_0}={\rm gcd}(m,r^{n'_0}-1) \mbox{ and } n_0={\rm ord}_{{\rm gcd}(m,r^{n_0}-1)}(r).\nonumber
    \end{equation}Let $d=n'_0\alpha$, where $\alpha\in\mathbb{N}^*$. Then the condition
    \begin{equation}
    r^d-1=(r^{n'_0}-1)\sum_{i=0}^{\alpha-1}r^{n'_0i}\equiv 0 \hspace{1mm}({\rm mod}\hspace{1mm}m),\nonumber
    \end{equation}implies
    \begin{equation}
    \sum_{i=0}^{\alpha-1}r^{n'_0i}\equiv 0 \hspace{1mm}({\rm mod}\hspace{1mm}m_0).\nonumber
    \end{equation}

    Assume that ${\rm gcd}(m_0,\frac{m}{m_0})\neq 1$ and take a prime $p$ dividing ${\rm gcd}(m_0,\frac{m}{m_0})$.
    Since $p\mid \frac{m}{m_0}$\,, it follows that
    \begin{equation}
    r^{n'_0}\equiv 1 \hspace{1mm}({\rm mod}\hspace{1mm}p),\nonumber
    \end{equation}and therefore
    \begin{equation}
    \sum_{i=0}^{\alpha-1}r^{n'_0i}\equiv \alpha \hspace{1mm}({\rm mod}\hspace{1mm}p).\nonumber
    \end{equation}On the other hand, by $p\mid m_0$ we have
    \begin{equation}
    \sum_{i=0}^{\alpha-1}r^{n'_0i}\equiv 0 \hspace{1mm}({\rm mod}\hspace{1mm}p).\nonumber
    \end{equation}Then $p\mid\alpha$, implying that $p\mid n$. One obtains $p\mid{\rm gcd}(m,n)$, a contradiction.
    Consequently,
    \begin{equation}
    {\rm gcd}(m_0,\frac{m}{m_0})=1.\nonumber
    \end{equation}

    It is now clear that the conditions a) and b) imply $m\mid r^{n_0}-1$, i.e. $d\mid n_0$, which together with the condition c) leads to
    \begin{equation}
    n_0=d.\nonumber
    \end{equation}

    Next, we will show that $m_0=1$. Assume $m_0\neq 1$ and denote $\beta={\rm ord}_{m_0}(r)$. Since both $m_0$ and $\frac{m}{m_0}$ divide $r^{\beta n'_0}-1$,
    we infer that $m\mid r^{\beta n'_0}-1$, and so $d\mid\beta n'_0$. Then
    \begin{equation}
    d\leq\beta n'_0\leq\varphi(m_0)n'_0< m_0n'_0,\nonumber
    \end{equation}which implies
    \begin{equation}
    m_{{\rm ZM}(m,n,r)}(H_{(m_0,n_0,s)})=\frac{m^2n^2}{dm_0n'_0}<\frac{m^2n^2}{d^2}=m_{{\rm ZM}(m,n,r)}(H_{(1,d,0)})\,,\nonumber
    \end{equation}contradicting the maximality of $m_{{\rm ZM}(m,n,r)}(H_{(m_0,n_0,s)})$. Hence $m_0=1$, as desired.
\end{proof}

We have now all ingredients to prove our main theorem.

\begin{proof}[Proof of Theorem \ref{th:C1}]
    The equality
    \begin{equation}
    m({\rm ZM}(m,n,r))=\frac{m^2n^2}{d^2}\nonumber
    \end{equation}follows by Lemma 4. Let $H_{(m_1,n_1,s)}\in {\cal CD}({\rm ZM}(m,n,r))$. Then
    \begin{equation}
    m_{{\rm ZM}(m,n,r)}(H_{(m_1,n_1,s)})=\frac{mn^2{\rm gcd}(m,r^{n_1}-1)}{m_1n_1{\rm ord}_{\frac{m}{m_1}}(r)}=\frac{m^2n^2}{d^2}\,,\nonumber
    \end{equation}or equivalently
    \begin{equation}
    d^2{\rm gcd}(m,r^{n_1}-1)=mm_1n_1{\rm ord}_{\frac{m}{m_1}}(r).\nonumber
    \end{equation}Since ${\rm gcd}(m,r^{n_1}-1)$, $m$ and $m_1$ are divisors of $m$, while $d$, $n_1$ and ${\rm ord}_{\frac{m}{m_1}}(r)$ are divisors of $n$,
    and ${\rm gcd}(m,n)=1$, we infer that
    \begin{equation}
    {\rm gcd}(m,r^{n_1}-1)=mm_1 \mbox{ and } d^2=n_1{\rm ord}_{\frac{m}{m_1}}(r).\nonumber
    \end{equation}Clearly, these equalities imply
    \begin{equation}
    m_1=1 \mbox{ and } n_1=d,\nonumber
    \end{equation}and by $s<m_0$ we obtain
    \begin{equation}
    s=0.\nonumber
    \end{equation}Hence $H_{(m_1,n_1,s)}=H_{(1,d,0)}$, completing the proof.
\end{proof}

\vspace*{5ex}\small

\hfill
\begin{minipage}[t]{5cm}
Marius T\u arn\u auceanu \\
Faculty of  Mathematics \\
``Al.I. Cuza'' University \\
Ia\c si, Romania \\
e-mail: {\tt tarnauc@uaic.ro}
\end{minipage}

\end{document}